\documentclass[11pt]{amsart}
\usepackage[right=3cm, left=3cm, top=2cm]{geometry}
\usepackage{enumerate, amssymb}
\newcommand{\rr}{\mathbb{R}}
\newcommand{\nn}{\mathbb{N}}
\newcommand{\zz}{\mathbb{Z}}
\newcommand{\cc}{\mathbb{C}}

\DeclareMathOperator{\newimaginarypart}{Im}

\renewcommand{\Im}{\newimaginarypart}

\DeclareMathOperator{\ind}{ind}

\newcommand{\<}{\left\langle}
\renewcommand{\>}{\right\rangle}

\newcommand{\Rmnum}[1]{\expandafter\@\romancap\romannumeral #1@}

\newtheorem{theorem}{Theorem}[section]
\newtheorem{prop}[theorem]{Proposition}
\newtheorem{cor}[theorem]{Corollary}
\newtheorem{lemma}[theorem]{Lemma}
\newtheorem{remark}[theorem]{Remark}

\theoremstyle{definition}
\newtheorem{mydef}[theorem]{Definition}

\renewcommand{\epsilon}{\varepsilon}

\DeclareMathOperator{\Log}{Log}

\DeclareMathOperator{\gauge}{G}
\DeclareMathOperator{\heisen}{Heis}
\newcommand{\fkX}{\mathfrak{X}}
\newcommand{\cale}{\mathcal{E}}
\newcommand{\B}{\mathfrak{B}}
\newcommand{\h}{\mathcal{H}}
\newcommand{\K}{\mathcal{K}}
\newcommand{\bh}{\B(\h)}
\newcommand{\tbeta}{\beta'}
\newcommand{\Ctb}{C_{\tbeta}}
\newcommand{\caln}{\mathcal{N}}

\newcommand{\calu}{\mathcal{U}}
\newcommand{\calo}{\mathcal{O}}
\newcommand{\torus}{\mathbb{T}}
\newcommand{\uh}{\calu(\h)}
\newcommand{\id}{\mathrm{id}}

\begin{document}

\title{Continuous families of E$_0$-semigroups}
\author{Ilan Hirshberg}
\author{Daniel Markiewicz}
\address{Department of Mathematics,
Ben-Gurion University of the Negev, P.O.B. 653, Be'er Sheva, 84105,
Israel.}
\email{ilan@math.bgu.ac.il}\email{danielm@math.bgu.ac.il}

\thanks{I.H. was partially supported by Israel Science Foundation grant 1471/07 and
D.M. was partially supported by U.S.-Israel Binational Science Foundation grant 2008295.
}

\begin{abstract}
We consider families of E$_0$-semigroups continuously parametrized by a compact Hausdorff space, which are cocycle-equivalent to a given E$_0$-semigroup $\beta$. When the gauge group of $\beta$ is a Lie group, we establish a correspondence between such families and principal bundles whose structure group is the gauge group of $\beta$. 
\end{abstract}
\maketitle

Let $\h$ be a Hilbert space, which we will always assume to be separable and infinite-dimensional, and let $\bh$ denote the $*$-algebra
of all bounded operators over $\h$. 
An \emph{$E_0$-semigroup acting on $\bh$} is a point-$\sigma$-weakly continuous family $\beta =\{\beta_t: \bh \to \bh\}_{t\geq0}$ of unital $*$-endomorphisms
such that $\beta_0 = \id$. We direct the reader to \cite{arv-monograph} for a general reference on the theory of E$_0$-semigroups.

This paper studies continuous families of E$_0$-semigroups parametrized by a compact Hausdorff space, where the E$_0$-semigroups are all cocycle equivalent to a given E$_0$-semigroup $\beta$. 
Suitable notions of continuity and equivalence are introduced below. If the gauge group $G$ of $\beta$ is a Lie group,  we show that such continuous families are classified by principal $G$-bundles. Gauge groups of E$_0$-semigroups were computed by Arveson in \cite{arv-analogues} for the type I case. In the type II case, the gauge groups were computed recently for several classes of examples by Alevras, Powers and Price in \cite{alevras-powers-price} and by Jankowski and the second author in \cite{jankowski-markiewicz}. Indeed, in many of the known examples the gauge group is a Lie group. Specializing to the case of E$_0$-semigroups of type I, one can recast those principal bundles as vector bundle invariants. The case of continuous families of single endomorphisms of $\bh$ (of finite index) was studied in \cite{hirshberg-endo-stable-cts-trace}, where it was shown that such families of endomorphisms are classified by vector bundle invariants of dimension given by the index of the endomorphism. We thus obtain an analogy between the case of continuous families of endomorphisms and continuous families of E$_0$-semigroups. 

In the case of families of one-parameter automorphism groups, the gauge group is $\rr$, hence the 
principal bundle invariants are trivial. We treat this case separately, since we establish this triviality result under (a priori) weaker continuity assumptions, using techniques from \cite{bargmann}. This corresponds to a parametrized version of Wigner's theorem.

Given an $E_0$-semigroup $\beta$ acting on $\bh$ we will say that a strongly continuous family of unitary operators
$\{ U_t  \in \uh: t\geq 0 \}$  is a \emph{$\beta$-cocycle} if $U_0=1$ and
$U_{t+s}=U_t\beta_t(U_s)$ for all $t,s \geq 0$. We emphasize that in this paper all  cocycles will be unitary cocycles. We will denote
by $C_\beta$ the set of all $\beta$-cocycles. An $E_0$-semigroup $\alpha$ is \emph{cocycle equivalent} to $\beta$
if there exists a $\beta$-cocycle $U_t$ such that $\alpha_t(X)=U_t \beta_t(X) U_t^*$ for all $t\geq 0$ and $X\in \bh$. We will denote by  $\cale_\beta$ be
the set of all $E_0$-semigroups acting on $\bh$ which are cocycle equivalent to the
$E_0$-semigroup $\beta$. 

If $\h$ is separable, then the unitary group $\uh$ is a Polish group when endowed with the relative strong operator topology. Recall that a Polish space is a topological space with a separable completely metrizable topology, and a Polish group is a topological group whose topology is Polish. Let $\rr_+$ denote the half-open interval $[0,\infty)$. Let us endow $C(\rr_+, \uh)$ with pointwise multiplication and the compact open topology. Since $\uh$ is Polish, we have that $C(\rr_+, \uh)$ is a Polish group (this can be proven using Theorem~4.19 of \cite{kechris-descriptive}).

Let $\beta$ be an E$_0$-semigroup acting on $\bh$. Every $\beta$-cocycle  $\{U_t : t \in \rr_+ \}$ can be seen as an element of $C(\rr_+, \uh)$, and we will always endow the set $C_\beta$ with the relative topology in $C(\rr_+,\uh)$. Since $C_\beta$ is clearly closed in $C(\rr_+,\uh)$, it follows that it is also a Polish space.

Recall that the gauge group $\gauge_\beta$ is the set of local (unitary) $\beta$-cocycles: $\gauge_\beta = \{ U \in C_\beta : U_t\beta_t(A)=\beta_t(A)U_t, \forall A \in \bh, \forall t\geq 0 \}$. Notice that $G_\beta$ is a closed subgroup of $C(\rr_+, \uh)$, hence it is a Polish group.

There is a natural (right) action of $G_\beta$ on $C(\rr_+, \uh)$ given by
$(U \cdot Q)_t = U_tQ_t$ for $U \in C(\rr_+, \uh)$, $Q \in G_\beta$. Furthermore,  $C_\beta$ is invariant under this action. Let $\pi: C_\beta \to \cale_\beta$ be given by $\pi(U) = \alpha$ where  $\alpha_t(A)=U_t \beta_t(A) U_t^*$ for all $A\in\bh$ and $t\geq 0$. One readily checks that the fibers of $\pi$ are exactly the orbits of the $G_\beta$ action.

We may thus endow $\cale_\beta$ with the quotient topology of $C_\beta/G_\beta$. If $\fkX$ is a compact Hausdorff space,  a continuous map $\alpha:\fkX \to \cale_\beta$ will be said to be a \emph{continuous family of E$_0$-semigroups parametrized by $\fkX$} in the cocycle-equivalence class of $\beta$. We shall also denote the set of all such continuous families by $\cale_\beta(\fkX)$. Given $\alpha \in \cale_\beta(\fkX)$, we shall denote by $\alpha^x$ the E$_0$-semigroup at the point $x \in \fkX$.

Let $\fkX$ be a compact Hausdorff space.  If $\theta\in \cale_\beta(\fkX)$ 
is a continuous family of E$_0$-semigroups, we shall say that a family 
$\{U^x : x \in  \fkX\}$ is a \emph{continuous family of $\theta$-cocycles} if 
$U^x$ is an $\theta^x$-cocycle for each $x \in \fkX$ and the map $x \mapsto U^x$ is 
continuous as a map $\fkX \to C(\rr_+,\uh)$.  We shall say that two continuous parametrized families $\theta, \sigma \in \cale_\beta(\fkX)$  are \emph{equivalent} if there exists a continuous family of $\sigma$-cocycles $U$ such that $\theta_t^x(A) = U_t^x \sigma_t^x(A) U_t^{x\,*}$ for all $t \in \rr_+$, $x \in \fkX$ and $A \in \bh$ (this is clearly an equivalence relation). Note that for any two families $\theta, \sigma \in \cale_\beta(\fkX)$ and for every $x \in\fkX$ it is possible to choose a $\theta^x$-cocycle $U^x$ such that $\sigma_x(A)=U^x_t \theta_t^x(A) U_t^{x\,*}$  for all $t\geq 0$ and $A\in\bh$, however it may not be possible to choose the cocycles in a continuous way with respect to $x$. We will say that a continuous family $\theta \in\cale_\beta(\fkX)$ is \emph{trivializable} if it is equivalent to the constant family equal to $\beta$ for all $x\in\fkX$.

We note that there are other natural notions of continuity for families of E$_0$-semigroups. We will say that the family of E$_0$-semigroups $\{\alpha^x \in \cale_\beta: x \in\fkX \}$ is \emph{$\sigma$-weakly continuous} if for any $A \in \bh$, the map $\fkX \times \rr_+ \to \bh$ given by $(x,t) \mapsto \alpha^x_t(A)$ is $\sigma$-weakly continuous. It is straightforward to check that any continuous family of E$_0$-semigroups is also $\sigma$-weakly continuous, however we do not know whether the converse holds. We note (Remark \ref{two-toplogies}) however, that for one-parameter automorphism groups the two notions do coincide, and thus it is conceivable that they coincide for E$_0$-semigroups as well.

\section{Preliminaries}
In this section we recall a few results concerning $G$-bundles. 

 Let $G$ be a locally compact group with identity $e$. A \emph{$G$-space} is a topological space $\fkX$ with a fixed continuous right action of $G$ (all our $G$-spaces will be right $G$-spaces).  We will say that $G$ acts \emph{freely} on $\fkX$ if for every $g \in G$, $g\neq e$ and for every $x\in \fkX$, $xg\neq x$.

 Given a free action of $G$ on $\fkX$, let $R=\{ (x,y) : \exists g\in G, xg=y \}$ denote the corresponding orbit equivalence relation endowed with the relative topology in $\fkX\times\fkX$. Since $G$ acts freely, given a pair $(x,y)\in R$ there exists a unique element $\Delta(x,y)\in G$ such that $y=x\Delta(x,y)$.

A free action of $G$ on a space $\fkX$ determines a $G$-bundle $\xi=(\fkX, \pi, \fkX/G)$ where $\pi:\fkX \to \fkX/G$ is the quotient map. We will say that $\xi$  is a \emph{principal G-bundle} if the map $\Delta:R\to G$ described above is continuous (we follow the terminology of \cite{husemoller}; this is called a Cartan principal $G$-bundle in \cite{palais}).

We emphasize that a principal $G$-bundle is not assumed to be locally trivial.  Under additional assumptions on $G$ and $\fkX$, however, local triviality is automatic.

\begin{theorem}[\cite{palais}, p.315]\label{thm-palais}
Let $G$ be a Lie group acting freely on a completely regular topological space $\fkX$, and let $\xi=(\fkX, \pi, \fkX/G)$ be the corresponding $G$-bundle. If $\xi$ is a principal $G$-bundle then it is locally trivial.
 \end{theorem}

We will say that a continuous action of a locally compact group $G$ on a completely regular space $X$ makes it a \emph{proper $G$-space} if every point $x \in \fkX$ has a neighborhood $U_x$ which is \emph{small} in the following sense:  each $y \in \fkX$ has a neighborhood $V_y$ such that $\{g \in G : U_x \cap (V_y\cdot g) \neq \varnothing\}$ has compact closure in $G$. We recall the following.

\begin{prop}[\cite{palais}, Proposition 1.2.8]\label{proper-quotient} If $\fkX$ is proper $G$-space, then $\fkX/G$ is completely regular.
\end{prop}

An open cover $\{U_\lambda \}_{\lambda \in \Lambda}$ for a topological space $B$ is \emph{numerable} provided that there exists a locally finite partition of unity $\{f_\lambda\}_{\lambda\in\Lambda}$ such that $\overline{f_\lambda^{-1}((0,1])} \subseteq U_\lambda$ for all $\lambda$.

A principal $G$-bundle $\xi=(E, p, B)$  is \emph{numerable} provided that there exists a numerable cover $\{U_\lambda\}_{\lambda \in \Lambda}$ for $B$ such that the restriction bundle $\xi|_{U_\lambda}$ is trivial for all $\lambda\in\Lambda$.

In particular, if  $\xi=(E, p, B)$ is  a locally trivial principal $G$-bundle and $B$ is  paracompact, then $\xi$ is numerable. We remark that a pullback of a numerable principal $G$-bundle is also numerable. For the following definition, see Chapter 4, Section 10 of \cite{husemoller}.

\begin{mydef} A numerable principal $G$-bundle $\xi_0=(E_0, p_0, B_0)$ is \emph{universal} if and only if
\begin{enumerate}
  \item For each numerable principal $G$-bundle $\xi = (E, p, B)$, there exists a continuous map $f: B \to B_0$ such that $\xi$ and the pull-back $f^*(\xi_0)$ are
      isomorphic;
  \item If $f,g : B \to B_0$ are two continuous maps such that $f^*(\xi_0)$ and $g^*(\xi_0)$ are isomorphic, then $f$ and $g$ are homotopic.
\end{enumerate}
\end{mydef}

The following result provides an effective criterion for establishing that a numerable principal bundle is universal.

\begin{theorem}[\cite{dold}, Theorem 7.5]\label{thm-dold}
A numerable principal $G$-bundle $\xi=(E, p, B)$ is universal if and only if $E$ is contractible.
\end{theorem}

\section{Continuous families of automorphism groups}

The purpose of this section is to show that any $\sigma$-weakly continuous family of one-parameter automorphism groups is trivializable. In effect, this amounts to a version of Wigner's theorem with a parameter. This can be obtained by modifying Bargmann's proof of Wigner's theorem from \cite{bargmann}.

\begin{mydef}
Let $G$ be an abelian topological group and let $0\in\caln\subseteq\rr$ be an open interval. We say that a continuous map $\omega: \caln\times\caln  \to G$ is a \emph{local $G$-valued multiplier} on $\caln$ if
\begin{align}
\label{mult-prop1}
\omega(t,0) & =  \omega(0, t) = e_G \\
\label{mult-prop2}
\omega(t,s) \omega(t+s,r) & = \omega(s,r) \omega(t, s+r)
\end{align}
for all $r,s,t \in \caln$ such that the required sums are also in $\caln$. 
If $G$ is a Lie group, we will say that a $G$-multiplier $\omega$ is \emph{smooth} if it is a $C^\infty$ function. 
Given a topological space $\fkX$ and a Lie group $G$, we will say that a family $\{ \omega^x : x \in \fkX \}$ of $G$-multipliers on $\caln$ is continuous if the map from $\fkX \times \caln \times \caln$ to $G$ given by $(x,t,s) \mapsto \omega^x(t,s)$ is continuous. We will say that the family is \emph{smooth} if $\omega^x$ is smooth for every $x \in \fkX$ and the map $(x,t,s) \mapsto \partial_t^n \partial_s^m \omega^x(t,s)$ is continuous, for any $n,m \in \zz_+$.

Let $\torus=\{z \in \cc: |z|=1\}$. We will refer to local $\torus$-valued multipliers simply as \emph{local multipliers}, and we will refer to $\rr$-valued multipliers as \emph{additive local multipliers}.
\end{mydef}

\begin{lemma}\label{get-unitaries-and-multiplier}
Let $\fkX$ be a compact Hausdorff space and let $\{ \alpha^x : x\in\fkX\}$ be a family of $*$-automorphisms of $\bh$ such that for all $A\in\bh$, the map $(x,t) \mapsto \alpha_t^x(A)$ is $\sigma$-weakly continuous. Then there exists an open interval $0 \in\caln \subseteq \rr$ and a family of unitaries $\{U^x_t : t \in \caln, x\in \fkX \}$ which is jointly strongly continuous in $(x,t)$ and which satisfies
\begin{equation}\label{eq:unit-impl}
\alpha_t^x(A) = U^x_t A {U_t^x}^*, \qquad \forall A \in \bh,\forall t\in\caln, \forall x\in\fkX.
\end{equation}
Furthermore, there exists a unique continuous family of local multipliers $\{\omega^x: x\in\fkX\}$ on $\caln$ which satisfies
$$
U_t^x U_s^x = \omega^x(t,s) U_{t+s}^x
$$
for all $x\in\fkX$ and $t,s\in\caln$ such that $t+s \in \caln$. Moreover, by shrinking $\caln$ if necessary, $\sigma^x = -i\Log \omega^x$ (the principal branch of $\log)$ is a well-defined bounded local additive multiplier on $\caln$ and $\{\sigma^x : x\in\fkX\}$ is a continuous family of additive multipliers on $\caln$.
\end{lemma}
\begin{proof}
Let $\xi \in \h$ be a unit vector, and let $p$ be the minimal projection onto the space spanned by $\xi$. Since $\alpha^x_t(p)$ is a strongly jointly continuous family of minimal projections and $\alpha_0^x(p)=p$, notice that $\{(x,t) \in \fkX \times\rr : \|\alpha^x_t(p)\xi\|>1/2\}$ is an open set containing $\fkX \times \{0\}$. Thus, by compactness, there exists $\delta>0$ such that 
$\| \alpha^x_t(p)\xi \| > 1/2$ for all $x\in\fkX$, $|t|<\delta$. 
By normalization we obtain a norm-continuous family of normalized vectors $\eta^x_t \in \h$ such that $\alpha^x_t(p)\eta_t^x = \eta_t^x$ for all $x \in \fkX$ and $|t|<\delta$. Now define for all $x\in\fkX$, $|t|<\delta$, the map $U_t^x: \h \to \h$ given by
$$
U_t^x A \xi = \alpha_t(A) \eta_t^x
$$
A direct computation shows that this map is a well-defined isometry, which is surjective hence unitary, and that equation \eqref{eq:unit-impl} holds for all $A\in \bh$ and $(x,t) \in \fkX\times (-\delta,\delta)$. We observe that this family is weakly continuous since for all $t,s \in (-\delta,\delta), x,y\in\fkX$, $A\in \bh$ and $\psi \in \h$,
\begin{align*}
|\< (U_t^x - U_s^y ) A\xi, \psi\>| & = 
|\< \alpha_t^x(A)\eta_t^x - \alpha_s^y(A)\eta_s^y, \psi \> | \\
& \leq | \< \alpha_t^x(A)\eta_t^x - \alpha_t^x(A)\eta_s^y, \psi \>| + 
| \< \alpha_t^x(A)\eta_s^y - \alpha_s^y(A)\eta_s, \psi \>|  \\
& \leq \| A \| \| \eta_t^x - \eta_s^y \| \|\psi\| + | \< \alpha_t^x(A)\eta_s^y - \alpha_s^y(A)\eta_s^y, \psi \> |
\end{align*}
and we are assuming that $\alpha_s^y(A)$ is $\sigma$-weakly continuous on $s,y$. Hence it is also strongly continuous since both topologies coincide on the unitary group. The verification of the remaining assertions of the lemma is straightforward so it will be omitted.
\end{proof}

The next two lemmas and the following theorem are a straightforward modification of Bargmann's work \cite{bargmann} on multipliers for the context of families with a parameter. For the reader's convenience we provide full details. Although we chose to focus on the case of actions of $\rr$ on $\bh$ by automorphisms, it is clear that analogous versions of the following results also hold for actions of Lie groups, by a similar modification of Bargmann's method.  In the following, given a set $S \subseteq \rr$, we will denote  $S+S = \{ t+s : t,s \in\ S \}$.

\begin{lemma}\label{diff-mult}
 Suppose that $\fkX$ is a compact Hausforff space. If $\{\sigma^x : x\in\fkX\}$ is a continuous family of bounded local additive multipliers on an open interval $0\in\caln\subseteq \rr$, then there exists a smooth family $\{ \zeta^x: x\in\fkX\}$ of bounded local additive multipliers on an open interval $0\in\caln'\subseteq \caln$  and a continuous function $\phi: \fkX \times (\caln'+\caln') \to \rr$ such that for all $x \in \fkX$ and $t,s \in \caln'$,
\begin{equation}\label{equiv1}
\zeta^x(t,s)  =  \sigma^x(t,s) -\phi(x, t+s) +\phi(x, t) + \phi(x, s), \quad \text{and} \quad
\phi(x,0) = 0.
\end{equation}
\end{lemma}
\begin{proof} 
Let $0\in\caln_1\subseteq\caln$ be an open interval such that $\caln_1 + \caln_1 \subseteq \caln$, and let $0 \in \caln' \subseteq\caln_1$ be an open interval such that $\caln' + \caln' \subseteq \caln_1$. Let $f:\rr \to \rr$ be a non-negative compactly supported $C^\infty$ function whose support is contained in $\caln'$ and moreover $\int_\rr f(x) dx =1$. Now define a function $a: \fkX \times \caln \to \rr$ by
$$
a(x,t) = - \int_\rr \sigma^x(t,r) f(r) dr.
$$
(The integrand vanishes outside $\caln'$). Then we have that $a(x,t)$ is continuous. Moreover, we have that the family $\{\theta^x : x\in\fkX\}$ given by
$$
\theta^x(t,s) = \sigma^x(t,s) - a(x,t+s) + a(x,t) + a(x,s), \qquad \forall t,s \in\caln_1
$$
is a continuous family of additive multipliers on $\caln_1$ since $a(x,0)=0$ for all $x\in\fkX$. Furthermore, notice that by using the multiplier property \eqref{mult-prop2},
\begin{align*}
\theta^x(t,s) & = \int_\rr \big[\sigma^x(t,s) + \sigma^x(t+s,r) - \sigma^x(t,r) - \sigma^x(s,r) \big] f(r) dr \\
& =  \int_\rr \big[\sigma^x(t,s+r) - \sigma^x(t,r) \big] f(r) dr \\
& =\int_\rr \sigma^x(t,u) \big[ f(u-s) - f(u) \big] du. 
\end{align*}
It follows that $\theta^x(t,s)$ is infinitely differentiable with respect to $s$. Now define $b:\fkX \times \caln_1 \to \rr$ by
$$
b(x,t) = - \int_\rr \theta^x(r,t) f(r) dr.
$$
Then we have that $b(x,t)$ is continuous and $b(x,0)=0$ for all $x\in\fkX$. By defining
$$
\zeta^x(t,s) = \theta^x(t,s) - b(x,t+s) + b(x,t) + b(x,s), \qquad \forall t,s \in\caln'
$$
we again obtain a continuous family $\{\zeta^x : x\in\fkX\}$ of additive multipliers on $\caln'$ and by computations analogous to those carried out above, we obtain for all $t,s\in\caln'$,
$$
\zeta^x(t,s) = \int_\rr \theta^x(u,s) \big[ f(u-t) - f(u) \big] du. 
$$
Thus for every $x\in\fkX$, $\zeta^x$ is a $C^\infty$ function, and it is clear that it constitutes a smooth family. Finally, note that \eqref{equiv1} is satisfied for $\phi(x,t) = a(x,t) + b(x,t)$.
\end{proof}

\begin{lemma}\label{trivial-mult}
Let $\fkX$ be a compact Hausdoff space. If $\{\zeta^x : x\in\fkX\}$ is a continuous family of bounded local additive multipliers on an open interval $0\in\caln\subseteq \rr$, then there exists an open interval $0 \subseteq \caln' \subseteq \caln$ and a continuous function $\phi: \fkX \times \caln \to \rr$ such that for all $x \in\fkX$ and $t,s \in\caln'$,
$$
\zeta^x(t,s) =  \phi(x, t+s) - \phi(x, t) - \phi(x, s), \qquad \text{and} \qquad \phi(x,0)=0.
$$
\end{lemma}
\begin{proof}
By the previous lemma, it suffices to prove the statement in the case that $\{ \zeta^x : x\in\fkX\}$ is a smooth family on $\caln$. Let us consider
$$
\theta(x, t, s) =\partial_s \zeta^x(t,s)
$$
which is a continuous function of $(x, t, s)$  since the family is smooth. 
Let $0 \in \caln'$ be an open interval  such that $\caln' + \caln' \subseteq \caln$. By differentiating the additive identity for $\zeta^x$ corresponding to multiplier identity \eqref{mult-prop2} with respect to $r$ at $r=0$, we obtain for all $x\in\fkX$ and $t,s \in \caln'$,
\begin{equation}\label{switch1}
\theta(x, t+s,0)  = \theta(x,s,0) + \theta(x, t, s).
\end{equation}
Define $\phi: \fkX \times \caln \to \rr$ by
$$
\phi(x,t) = \int_0^t \theta(x, u, 0) du  
$$
It is clear that $\phi$ is continuous and $\phi(x,0)=0$ for all $x\in\fkX$. Moreover, 
for all $u,v \in\caln'$,
\begin{align*}
\phi(x,u+v) - \phi(x,u) - \phi(x,v) 
& = \int_u^{u+v} \theta(x,r,0)dr - \int_0^v \theta(x,s,0) ds \\
& = \int_0^{v} \theta(x,u +s,0)du - \int_0^s \theta(x,s,0) ds \\
\text{by \eqref{switch1}}& = \int_0^{v} \theta(x,u,s) ds  \\
&= \int_0^{v} \partial_s\zeta^x(u,s) ds  \\
& = \zeta^x(u,v) - \zeta^x(u,0) \\
\text{ by \eqref{mult-prop1}}&= \zeta^x(u,v). 
\end{align*}
\end{proof}

We now obtain a version of Wigner's theorem with a parameter.

\begin{theorem}\label{wigner-parameter}
Let $\fkX$ be a compact Hausdorff space and let $\{ \alpha^x : x\in\fkX\}$ be a family of $*$-automorphisms of $\bh$ such that for all $A\in\bh$, the map $(x,t) \mapsto \alpha_t^x(A)$ is $\sigma$-weakly continuous. Then there exists a family of unitaries $\{U^x_t : t \in \rr, x\in \fkX \}$ in $\bh$ such that
\begin{enumerate}
\item $(x,t) \mapsto U_t^x$ is strongly continuous
\item for each $x \in\fkX$, $\{ U_t^x: t\in \rr\}$ is a unitary group
\item $\alpha_t^x(A) = U^x_t A {U_t^x}^*,$ for all  $A \in \bh$, $t\in\rr$ and  $x\in\fkX$.
\end{enumerate}
In particular, the family $\{ \alpha^x : x\in\fkX\}$ is  trivializable.
\end{theorem}
\begin{proof} By Lemma~\ref{get-unitaries-and-multiplier}, there exists an open interval $0\in\caln$ and a family $\{V_t^x : t\in \caln, x\in\fkX\}$ implementing the family $\{\alpha^x : x\in\fkX\}$, and a continuous family of additive multipliers $\{\sigma^x : x\in\fkX\}$ on $\caln$ such that
$$
V_t^x V_s^x = \exp(-i\sigma^x(t,s)) V^x_{t+s},
$$
for all $t,s\in\caln$ such that $t+s \in\caln$ and $x\in\fkX$. Furthermore, by Lemmas~\ref{diff-mult} and \ref{trivial-mult}, perhaps by shrinking $\caln$, there exists a continuous function $\phi:\fkX\times \caln \to \rr$ and an open interval $0\in \caln' \subseteq \caln$ such that for all $x\in\fkX$ and $t,s\in\caln'$,
$$
\sigma^x(t,s) =  \phi(x, t+s) - \phi(x, t) - \phi(x, s), \qquad \text{and} \qquad \phi(x,0)=0.
$$
In particular, if we define $U_t^x= e^{-i\phi(x,t)}V_t^x$ for $x\in\fkX$ and $t\in\caln$ we obtain a new family of unitary operators which is strongly continuous in $(x,t)$,  which implements the automorphisms $\{\alpha^x : x\in\fkX\}$ for $t\in \caln$, and which furthermore satisfies
$$
U_t^x U_s^x = U_{t+s}^x, \qquad  \forall t,s \in\caln', \forall x\in\fkX.
$$
It is clear that we now wish to define for all $t\in \rr$, $U_t^x = (U_{t/n}^x)^n$ where $n\in\nn$ is large enough so that $t/n \in\caln'$. This gives rise to a well-defined family of operators, because for all $x\in\fkX$ and  $t_1, \dots t_m, s_1,\dots s_n \in \caln'$ such that $\sum_{j=1}^n t_i = \sum_{k=1}^m s_m$, it is easy to see that
$$
U^x_{t_1} U^x_{t_2} \dots U^x_{t_n} = U^x_{s_1}U^x_{s_2}\dots U^x_{s_m}.
$$
The strong continuity of the family $\{U_t^x\}$ follows from the continuity of $\phi$ and the family $V_t^x$ on $(x,t)$.
\end{proof}

Despite the foregoing result, by itself the notion of point-$\sigma$-weak continuity for a family of E$_0$-semigroups $\{ \alpha^x : x \in \fkX \}$ of $\bh$ over a compact Hausdoff space $\fkX$ is not sufficiently stringent for applications. Indeed, cocycle-conjugacy classes need not be preserved under continuous parametrizations in this sense. For an example of this phenomenon, let $\beta$ be any E$_0$-semigroup acting on $\bh$ which is not cocycle conjugate to an automorphism group, and consider the family 
$\alpha = \{ \alpha^x : x \in [0,1] \}$ where
$$
\alpha^x_t(A) := \beta_{xt}(A), \qquad \forall A \in \bh, t \geq 0.
$$
It is clear that $\alpha^x$ is an E$_0$-semigroup for every $x \in [0,1]$, and for 
every $A\in\bh$ the map $(x,t) \mapsto \alpha^x_t(A)$ is $\sigma$-weakly continuous. 
Notice however that $\alpha^0_t(x)=x$ for all $t>0$ and $x\in\bh$. 

In this context, it seems that the appropriate reading of Theorem~\ref{wigner-parameter} is that it provides a lift of a continuous family of E$_0$-semigroups which are \emph{in the same cocycle-cojugacy class}.

\begin{remark}\label{two-toplogies}
We note that it follows immediately from Theorem \ref{wigner-parameter} that any $\sigma$-weakly continuous family of automorphism groups is a continuous family in our sense.  
\end{remark}

\section{Continuous families of E$_0$-semigroups}

We first recall the following theorem concerning invariant metrics. For a proof, see Theorem~8.3, p.70 of \cite{hewitt-ross1}.

\begin{theorem}[Birkhoff-Kakutani]
Any metrizable topological group admits a left-invariant metric. 
\end{theorem}

The Polish group of interest to us here is the group $C(\rr_+,\uh)$. We note  that we can write such a metric explicitly as follows. Let $(\psi_k)_{k\in \nn}$ be a dense sequence in the unit sphere of $\h$. For any two elements $U,V \in C(\rr_+,\uh)$, define 
$$
d(U,V) = \sum_{k=1}^{\infty} \frac{1}{2^k} \sup_{t \in [0,k]} \|U_t\psi_k - V_t\psi_k\|
$$
The exact form of the metric is of no importance in the sequel, however. 

\begin{prop} \label{proper-G-space-lemma}
Let $H$ be a Polish group, and let 
$G$ be a  subgroup of $H$. If $X$ is any subset of $H$ which is $G$-invariant, then
$G$ acts freely and continuously on $X$. If in addition $G$ is locally compact (in the relative topology), then $X$ is a proper (right) $G$-space and $(X, \pi, X/G)$ is a principal $G$-bundle. 
\end{prop}
\begin{proof}
It is clear that $G$ acts continuously and freely on $H$, hence it acts continuously and freely on $X$. Now suppose that it is locally compact.

Let $d$ be a left invariant metric on $H$. We denote by $B_{r}(h)$ the ball with center $h$ and radius $r$ in $H$. Given $x \in X$, we denote $B_r^X(x) = B_r(x) \cap X$. Since $G$ is locally compact, it follows that there is some $r_0>0$ such that for any $0<r<r_0$, any set of diameter $4r$ in $G$ has compact closure in $G$, and in particular $B_{2r}(h) \cap G$ has compact closure in $G$ for any $h \in H$. Fix such an $0<r<r_0$.

We prove that $X$ is a proper $G$-space by showing that for every $x \in X$ the neighborhood $B_r^X(x)$ is small. Fix $y \in X$, and let $V_y$ be an open neighborhood of $y$ in $X$ such that for all $z \in V_y$, $d(z^{-1}x, y^{-1}x) < r$. This is possible by continuity of the operations in $H$. We now claim that the set 
$$
A_{x,y} = \{g \in G : B_r^X(x) \cap V_y \cdot g \neq \varnothing\} 
$$
has a compact closure in $G$. To see that, we note that $g \in A_{x,y}$ if and only if there exists $z\in V_y$ such that $zg \in B_r^X(x)$, so $d(g, z^{-1}x) = d(zg,x) < r$ by left-invariance. Thus, $d(g,y^{-1}x)<2r$, so $A_{x,y} \subseteq B_{2r}(y^{-1}x)\cap G$, and therefore has compact closure in $G$. 

Since the action of $G$ on $X$ is free, in order to prove that $X$ is a principal $G$-bundle it remains to show that the function $\Delta$ uniquely defined on the equivalence relation $R= \{ (x,y) \in X \times X : \exists g \in G, y=xg \}$ by
$$
x \Delta(x,y) = y, \qquad \forall (x,y) \in R
$$
is continuous. But this is evident since $\Delta(x,y)=x^{-1}y$. 
\end{proof}

We obtain the following immediate corollary by taking $H = C(\rr_+,\uh)$, $X = C_\beta$ and $G = G_\beta$.

\begin{cor} \label{cor:proper-bundle}
If $\beta$ is an E$_0$-semigroup acting on $\bh$, then its gauge group $\gauge_\beta$ acts freely on the cocycle space $C_\beta$ via right multiplication: for all $U \in C_\beta$ and $Q \in \gauge_\beta$, $(U \cdot Q)_t := U_t Q_t$, $t\geq 0.$
Furthermore, if $G_\beta$ is locally compact, then $C_\beta$ is a proper $G_\beta$-space and $\Xi_\beta=(C_\beta, \pi, \cale_\beta)$ is a principal $G_\beta$-bundle, where $\pi:C_\beta \to \cale_\beta$ denotes the natural quotient map onto $\cale_\beta$ endowed with the quotient topology. 
\end{cor}

\begin{prop}\label{contractive}
If $\beta$ is an E$_0$-semigroup acting on $\bh$, then $C_\beta$ is contractible.
\end{prop} 
\begin{proof}
Consider the $E_0$-semigroup $\tbeta$ on $\B(\h \otimes L^2(\rr_+))$
given by $\tbeta = \beta \otimes id_{\B(L^2(R_+))}$. Since $\beta$
is cocycle conjugate to $\tbeta$, it suffices to show that $\Ctb$ is
contractible. Let $\{V_\lambda: \lambda\geq 0 \}$ be the right shift semigroup 
on $L^2(\rr_+)$ and let $S_\lambda = 1_{\bh} \otimes V_\lambda$, $\lambda\geq 0$ 
 be the right shift semigroup on $\h \otimes L^2(\rr_+)$.
Denote $P_\lambda = 1 - S_\lambda S_\lambda^* = 1 \otimes (1 - V_\lambda V_\lambda^*)$ for all $\lambda\geq 0$.  Observe 
that for all $t\geq 0$ and $\lambda\geq 0$,
$\tbeta_t(S_\lambda) = S_\lambda$ and $\tbeta_t(P_\lambda) = P_\lambda$.

Let us define $F: [0,\infty) \times \Ctb \to \Ctb$
by $F(\lambda,U)_t = P_\lambda + S_\lambda U_t S_\lambda^*$. We need to check that 
the range of $F$ is in $\Ctb$ (once that is done, it is clear that 
$F$ is continuous). Given a $\tbeta$-cocycle $U$, we must show
that for all $\lambda\geq 0$, $F(\lambda,U)$ is also a cocycle, i.e. we should
check that $F(\lambda,U)_t \tbeta_t(F(\lambda,U)_s) = F(\lambda,U)_{t+s}$, and that
$F(\lambda,U)$ is unitary. Indeed,
\begin{align*}
F(\lambda,U)_t\; \tbeta_t(F(\lambda,U)_s) & =
(P_\lambda + S_\lambda U_t S_\lambda^*)\tbeta_t \big(P_\lambda + S_\lambda U_s S_\lambda^*\big) \\
& =(P_\lambda + S_\lambda U_t S_\lambda^*) (P_\lambda + S_\lambda \tbeta_t(U_s) S_\lambda^*) \\
& =P_\lambda + S_\lambda U_t  \tbeta_t(U_s) S_\lambda^* 
= P_\lambda + S_\lambda U_{t+s} S_\lambda^* \\
&= F(\lambda,U)_{t+s} 
\end{align*}
thus $F(\lambda,U)$ satisfies the cocycle condition. To check that it is
unitary, we write
$$
F(\lambda,U)_t F(\lambda,U)_t^* = (P_\lambda + S_\lambda U_t S_\lambda^*)(P_\lambda + S_\lambda U_t^* S_\lambda^*)
= P_\lambda + S_\lambda U_t U_t^* S_\lambda^* = 1
$$
and checking that $F(\lambda,U)_t^* F(\lambda,U)_t = 1$ is similar. Now note that for any sequence $(\lambda_n, U_n)$, such that $\lambda_n \to \infty$, we have that $F(\lambda_n,U)$ converges to 1 as $n \to \infty$, hence we have a well-defined continuous extension of $F$ to $[0,\infty]\times \Ctb$ which is a contraction as required.
\end{proof}

This allows us to prove the following theorem. 

\begin{theorem}
Let $\beta$ be an E$_0$-semigroup such that its gauge group $\gauge_\beta$ is a Lie group. Then $\Xi_\beta=(C_\beta, \pi, \cale_\beta)$ is a universal principal $G_\beta$-bundle.
\end{theorem}
\begin{proof} 
In view of Corollary~\ref{cor:proper-bundle} and Proposition~\ref{contractive},
 in order to conclude that the bundle $\Xi_\beta$ is universal,  by 
Theorem~\ref{thm-dold} it only remains to show that $\Xi_\beta$ is numerable.  
Since $C_\beta$ is completely regular (it is Polish) and $G_\beta$ is a Lie group, 
we have that $\Xi_\beta$ is locally trivial by Theorem~\ref{thm-palais}. Thus 
to complete the proof it suffices to show that $\cale_\beta$ is paracompact. We will show that $\cale_\beta$ is second countable and regular, hence metrizable by Urysohn's metrization theorem, and metrizability implies paracompactness. It follows from Proposition~\ref{proper-quotient} and Corollary~\ref{cor:proper-bundle}
 that $\cale_\beta$ is completely 
regular. Thus it remains to establish that it is second countable. It suffices to 
show that the quotient map $\pi: C_\beta \to \cale_\beta$ is open, since  $C_\beta$ is second countable, as it is in fact Polish. That $\pi$ is open is a consequence of the local triviality of  the bundle, as follows. Let $\calo \subseteq C_\beta$ be an open subset. Let 
$U \in \calo$ be a given cocycle. It suffices to show that the point $\pi(U)$ has a
 neighborhood which is contained in $\pi(\calo)$. By local triviality, $\pi(U)$ has 
a neighborhood $\caln$ such that $\pi^{-1}(\caln)$ is homeomorphic to 
$\caln \times G_\beta$ via 
a homeomorphism which conjugates $\pi$ to the first coordinate projection. In
 particular, the restriction of $\pi$ to $\pi^{-1}(\caln)$ is open. Thus, 
$\pi(\calo \cap \pi^{-1}(\caln))$ is an open neighborhood of $U$ contained in 
$\pi(\calo)$, as required.
\end{proof}

Thus for continuous families of E$_0$-semigroups in the cocycle conjugacy class of $\beta$, we obtain the following.

\begin{cor} \label{spelling-out-cor}
Let $\beta$ be an E$_0$-semigroup whose gauge group is a Lie group. Let $\fkX$ be a compact Hausdorff space. For every continuous family $\alpha \in\cale_\beta$, define the locally trivial principal $G_\beta$-bundle over $\fkX$ given by the pullback $\xi(\alpha):=\alpha^*(\Xi_\beta)$. For all $\alpha, \gamma \in\cale_\beta(\fkX)$,
\begin{enumerate}
\item\label{1} $\xi(\alpha) \cong \xi(\gamma)$ if and only if  $\alpha$ is equivalent to $\gamma$.
\item\label{2} $\xi(\alpha)$ is trivial if and only if the family $\alpha$ is trivializable. 
\item\label{3} Any locally trivial principal $G_\beta$-bundle over $\fkX$ is isomorphic to $\xi(\alpha)$ for some  $\alpha \in \cale_\beta(\fkX)$.
\end{enumerate}
\end{cor}
\begin{proof}
For (\ref{1}), it is clear that if $\alpha$ and $\gamma$ are equivalent then $\xi(\alpha) \cong \xi(\gamma)$. For the converse, let $j: \xi(\alpha) \to \xi(\gamma)$ be an isomorphism. We define a continuous family of $\gamma$-cocycles $\{V^x : x\in \fkX\}$ as follows: for each $t \geq 0$, $x\in\fkX$, let $U_t^x$ be any element in the fiber of $\xi(\alpha)$ over $x$. We let 
$$
V_t^x = U_t^x j(U_t^x)^*
$$
We note that $V_t^x$ does not depend on the choice of $U_t^x$, because $j$ is $G$-equivariant: if $W_t^x$ is another element of the fiber of $\xi(\alpha)$ over $x$, then there exists $Q \in G_\beta$ such that $W^x_t = U_t^x Q_t$ and 
$ W_t^x j(W_t^x)^* =  U_t^x Q_t j(U_t^x Q_t)^* =  U_t^x Q_t Q_t^* j(U_t^x)^* =V_t^x$. It is straightforward to verify that $V^x$ is a $\gamma^x$-cocycle. Continuity now follows from local triviality of $\xi(\alpha)$. 

Part \eqref{2}  follows immediately from item \eqref{1}, and part \eqref{3} follows directly from universality of the bundle $\Xi_\beta$.
\end{proof}

We remark that, in general, the gauge group of an E$_0$-semigroup need not be locally compact. For example, let us consider E$_0$-semigroups of type I. 

 Let $\K$ be a separable Hilbert space, which is allowed to have dimension zero. We consider the Heisenberg group $\heisen(\K) = \{ (c,\xi) : c\in \rr, \xi \in \K\}$ with the product topology and the operations
 \begin{align*}
 (c_1, \xi_1) \cdot (c_2, \xi_2) & = (c_1 +c_2 + \Im \langle \xi_1, \xi_2 \rangle,\; \xi_1 + \xi_2) \\
 (c, \xi)^{-1} & = ( -c, -\xi)
 \end{align*}
 The unitary group $\calu(\K)$ acts on $\heisen(\K)$ by automorphisms as follows,
 $$
 U \cdot (c, \xi) = (c, U\xi)
 $$
 hence we have a semidirect product $\heisen(\K) \rtimes \calu(\K)$. This is a Polish group with the product topology (here $\calu(\K)$ has the strong operator topology as usual). Its operations can be
 described explicitly by the following:
\begin{align*}
(c_1, \xi_1, U_1) \cdot (c_2, \xi_2, U_2) & = (c_1 + c_2 + \Im \langle \xi_1, U_1\xi_2 \rangle, \; \xi_1 + U_1 \xi_2, U_1 U_2 )\\
(c, \xi, U)^{-1} & = ( -c , -U^{-1}\xi, U^{-1})
\end{align*}
Naturally, when $\dim \K=0$, we mean $\heisen(\K) \rtimes \calu(\K)=\rr$. 

\begin{theorem}[\cite{arv-analogues}]\label{typeI-gauge-group}
Let $\beta$ be an E$_0$-semigroup of type I, and let $\K$ be a Hilbert space such that $\dim \K = \ind \beta$. Then $\gauge_\beta$ is isomorphic as a topological group to $\heisen(\K) \rtimes \calu(\K)$.
\end{theorem}

It follows that in the case $\ind \beta = \infty$ the gauge group is manifestly not locally compact. Nevertheless, Theorem~\ref{typeI-gauge-group} implies that when the gauge group of an E$_0$-semigroup of type I is locally compact then it is a Lie group. That is also the case for many of the E$_0$-semigroups of type II$_0$ analyzed recently by \cite{alevras-powers-price} and \cite{jankowski-markiewicz}. In fact, to our knowledge in all cases when the gauge group has been proven to be locally compact, it has also turned out to be a Lie group.

Suppose $\dim (\K) = n < \infty$. Since $\heisen(\K)$ is contractible,  $\heisen(\K) \rtimes \calu(\K)$-bundles are prolongations of $\calu(\K)$-bundles (see Chapter~6 of \cite{husemoller}). In turn there is a functorial  bijection between $\calu(\K)$-bundles and the corresponding associated $n$-dimensional (complex) vector bundles. Thus for the case of continuous families of E$_0$-semigroups of type I$_n$, Corollary \ref{spelling-out-cor} is equivalent to the following form. Note the analogy with Theorem~1 of \cite{hirshberg-endo-stable-cts-trace}. 

\begin{cor} 
Suppose that $\beta$ is an E$_0$-semigroup of type I$_n$, and let $\fkX$ be a compact Hausdorff space. To each continuous family $\alpha \in\cale_\beta(\fkX)$ we have an associated $n$-dimensional vector bundle $v(\alpha)$ over $\fkX$ such that the following hold. For all $\alpha, \gamma \in\cale_\beta(\fkX)$,
\begin{enumerate}
\item $v(\alpha) \cong v(\gamma)$ if and only if $\alpha$ is  equivalent to $\gamma$.
\item $v(\alpha)$ is trivial if and only if the family $\alpha$ is trivializable. 
\item Any $n$-dimensional vector bundle over $\fkX$ is isomorphic to $v(\alpha)$ for some  $\alpha \in\cale_\beta(\fkX)$.
\end{enumerate}
\end{cor}

 We note that if $G_\beta$ is a Lie group, then it follows that a $\sigma$-weakly continuous family is a continuous family if and only if it is locally trivial. We have seen that local triviality (indeed, triviality) is automatic in the case of automorphism groups, but we do not know if this is the case for E$_0$-semigroups. 
However, it follows from the corollaries above that there exist topological obstructions for having a parametrized version of Wigner's theorem in the $\sigma$-weakly continuous setting. For instance, we have the following.

\begin{cor}
If $\fkX$ is a compact Hausdorff space which admits a non-trivial $n$-dimensional complex vector bundle, then there exists a (locally trivial) $\sigma$-weakly continuous family of E$_0$-semigroups of type I$_n$ parametrized by $\fkX$ which is not trivializable. 
\end{cor}

\providecommand{\bysame}{\leavevmode\hbox to3em{\hrulefill}\thinspace}


\begin{thebibliography}{APP06}

\bibitem[APP06]{alevras-powers-price}
A.~Alevras, R.~T. Powers, and G.~L. Price, \emph{Cocycles for one-para\-meter
  flows of {$B(H)$}}, J. Funct. Anal. \textbf{230} (2006), no.~1, 1--64.

\bibitem[Arv89]{arv-analogues}
W.~Arveson, \emph{Continuous analogues of {F}ock space}, Mem. Amer. Math. Soc.
  \textbf{80} (1989), no.~409, iv+66.

\bibitem[Arv03]{arv-monograph}
\bysame, \emph{Noncommutative dynamics and {$E$}-semigroups}, Springer
  Monographs in Mathematics, Sprin\-ger-Ve\-rlag, New York, 2003.

\bibitem[Bar54]{bargmann}
V.~Bargmann, \emph{On unitary ray representations of continuous groups}, Ann.
  of Math. (2) \textbf{59} (1954), 1--46.

\bibitem[Dol63]{dold}
A.~Dold, \emph{Partitions of unity in the theory of fibrations}, Ann. of Math.
  (2) \textbf{78} (1963), 223--255.

\bibitem[Hir04]{hirshberg-endo-stable-cts-trace}
I.~Hirshberg, \emph{Endomorphisms of stable continuous-trace {$C^*$}-algebras},
  Proc. Amer. Math. Soc. \textbf{132} (2004), no.~2, 481--486 (electronic).

\bibitem[HR79]{hewitt-ross1}
E.~Hewitt and K.~A. Ross, \emph{Abstract harmonic analysis. {V}ol. {I}}, second
  ed., Grundlehren der Mathematischen Wissenschaften, vol. 115,
  Springer-Verlag, Berlin, 1979.

\bibitem[Hus94]{husemoller}
D.~Husemoller, \emph{Fibre bundles}, third ed., Graduate Texts in Mathematics,
  vol.~20, Springer-Verlag, New York, 1994.

\bibitem[JM11]{jankowski-markiewicz}
C.~Jankowski and D.~Markiewicz, \emph{Gauge groups of {E}$_0$-semigroups
  obtained from {P}owers weights}, preprint arxiv:1102.5671, 2011.

\bibitem[Kec95]{kechris-descriptive}
A.~S. Kechris, \emph{Classical descriptive set theory}, Graduate Texts in
  Mathematics, vol. 156, Springer-Verlag, New York, 1995.

\bibitem[Pal61]{palais}
R.~S. Palais, \emph{On the existence of slices for actions of non-compact {L}ie
  groups}, Ann. of Math. (2) \textbf{73} (1961), 295--323.

\end{thebibliography}
\end{document}